\documentclass{amsart}
\usepackage{pifont}
\usepackage{amsfonts}

\usepackage{amscd,amssymb,amsmath,graphicx,verbatim}
\usepackage[dvips]{hyperref}
\usepackage[TS1,OT1,T1]{fontenc}

\newtheorem{theorem}{Theorem}[section]
\newtheorem{lemma}[theorem]{Lemma}
\newtheorem{corollary}[theorem]{Corollary}
\newtheorem{proposition}[theorem]{Proposition}

\theoremstyle{definition}
\newtheorem{definition}[theorem]{Definition}
\newtheorem{example}[theorem]{Example}

\newtheorem{claim}[theorem]{Claim}

\theoremstyle{remark}
\newtheorem{remark}[theorem]{Remark}

\numberwithin{equation}{section}

\begin{document}

\title{Schauder Bases and Operator Theory}

\author{Yang Cao}
\address{Yang Cao, Department of Mathematics , Jilin university, 130012, Changchun, P.R.China} \email{Caoyang@jlu.edu.cn}

\author{Geng Tian}
\address{Geng Tian, Department of Mathematics , Jilin university, 130012, Changchun, P.R.China} \email{tiangeng09@mails.jlu.edu.cn}

\author{Bingzhe Hou}
\address{Bingzhe Hou, Department of Mathematics , Jilin university, 130012, Changchun, P.R.China} \email{houbz@jlu.edu.cn}

\date{Oct. 14, 2010}
\thanks{}
\begin{abstract}
In this paper, we firstly give a matrix approach to the bases of a separable Hilbert space and then correct a mistake
appearing in both review and the English translation of the Olevskii's paper. After this, we show that even a diagonal compact
operator may map an orthonormal basis into a conditional basis.
\end{abstract}
\maketitle

\section{Introduction and preliminaries}
In operator theory, an invertible
operator on an infinite dimensional complex Hilbert space $\mathcal{H}$
 means the bounded operator which has a bounded inverse operator, and it is well-known that, for an $n\times n$ matrix $M_{n}$ (seen as an operator on finite dimensional Hilbert space $\mathbb{C}^n$), $M_{n}$ is invertible if and only if its column vectors are linearly
independent in $\mathbb{C}^{n}$. In other words, the column vectors of $M_{n}$ comprise a basis of $\mathbb{C}^{n}$.
From this point of view, we could generalize the "invertibility" of $\omega \times \omega$
matrix $M$ (the representation of a bounded operator on an orthonormal basis of $\mathcal{H}$) in the following manner: all column vectors of $M$ form some kind of basis of $\mathcal{H}$.  Actually,
the invertible operator do have a natural understanding in the `basis' language. That is,
the column (or row) vectors of the matrix of an invertible operator always comprise a `Riesz basis'
(it is a direct corollary of theorem 2, paper \cite{J-M}, although the authors do not state it in this way).
From the above observation,
it suggests us to consider the $\omega \times \omega$ matrix whose column vectors form more general kind of bases.

Naturally we consider the $\omega \times \omega$ matrix whose column vectors comprise a Schauder basis. We shall call them
the \textsl{Schauder matrix} therefrom. An operator which has a Schauder matrix representation under some orthonormal basis (ONB) will be called a
{\textsl{Schauder operator}. An easy fact is that an operator is a Schauder operator if and only if it maps some ONB into a Schauder basis.
Many scholars have studied some kind of these operators. A. M. Olevskii gave a surprising result on the bounded operators which map some ONB
into a conditional quasinormal basis (\cite{Ole}, theorem 1, p479); Stephane Jaffard and Robert M. Young proved
that a Schauder basis always can be given
by an one-to-one positive transformation (\cite{J-M}, theorem 1, p554). I. Singer gave lots of examples of bases of $\mathcal{H}$
which can be rewritten into a matrix form (see, \cite{Singer}, p429, p497). Besides these results, as for
a joint research both on operator theory and the basis theory but not in this direction, the paper \cite{Gowers}, \cite{Gowers2} by
Gowers, the paper \cite{Kwapien} by
Kwapien, S. and Pelczynski, A. and the elegant book \cite{Y2} by M. Young
are remarkable examples.

Nevertheless, there is still a gap between the researches in the field of basis theory and operator theory. There are few
joint works on both basis of Hilbert space and the operators on the Hilbert space. The reason reflects on two aspects. One is the different terminology systems and the other one is
that there are scanty common objects to study with. The main purpose of this paper is to show that the Schauder matrix
is a candidate to fill this gap. As basic and traditional tools, the matrix representation of operators plays
an important role in the study of the operators on the Hilbert space $\mathcal{H}$. So the matrix  approach to the
basis theory is a good beginning to the joint research on the bases of the Hilbert space $\mathcal{H}$ and the operators on it.

In this paper, the matrix representation of operators and bases will be the bridge between basis theory and operator theory.
We firstly give a matrix approach to the bases of a separable Hilbert space and then correct a mistake
appearing in both review and the English translation of the Olevskii's paper.
After this, we follow the Olevskii's result to
consider the operators which can map some ONB into
a conditional Schauder basis. We shall call them \textsl{conditional operators} therefrom.
In matrix language, it is equivalent to study the operator $T$ which has a matrix representation $M$ under some ONB such that
the column vector sequence of $M$ comprise a conditional Schauder basis.
\section{An Operator Theory Description of Schauder basis}

\subsection{}\label{Subsec:Matrix Form of Basis}

Suppose that $\{e_{k}\}_{k=1}^{\infty}$ is an ONB of $\mathcal{H}$. An $\omega \times \omega$ matrix
$M=(m_{ij})$ automatically represents an operator under this ONB. In more details, for a vector $x \in \mathcal{H}$ there is
an unique $l^{2}-$sequence $\{x_{n}\}_{n=1}^{\infty}$ such that $x=\sum_{n=1}^{\infty}x_{n}e_{n}$ in which the
series converges in the norm of $\mathcal{H}$. Let $y_{n}=\sum_{k=1}^{\infty} m_{ik}x_{k}, y=\sum_{n=1}^{\infty}y_{n}e_{n}$,
then the operator $T_{M}$ defined by $T_{M}x=y$ is just the corresponding operator represented by $M$. In general, $T$
is not a bounded operator. We shall
identify the $\omega \times \omega$ matrix $M$ and the operator $T_{M}$, and denote them by the same notation $M$
if we have fixed an ONB and there is no confusion.

Recall that a sequence $\psi=\{f_{n}\}_{n=1}^{\infty}$ is called a \textsl{Schauder basis} of the Hilbert space $\mathcal{H}$
if and only if for every vector $x \in \mathcal{H}$ there exists an unique sequence $\{\alpha_{n}\}_{n=1}^{\infty}$ of complex numbers
such that the partial sum sequence $x_{k}=\sum_{n=1}^{k} \alpha_{n}f_{n}$ converges to $x$ in norm.

Denote by $P_{k}$ the the diagonal operator with the first $k-$th entries on diagonal line equal to 1 and $0$ for others.
Then as an operator $P_{k}$ represents the orthogonal projection
from $\mathcal{H}$ to the subspace $\mathcal{H}^{(k)}=span \{e_{1}, e_{2}, \cdots, e_{k}\}$.

\begin{lemma}\label{Matrix Form}
Assume that $\{e_{k}\}_{k=1}^{\infty}$ is a fixed ONB of $\mathcal{H}$.
Suppose that an $\omega \times \omega$ matrix $F=(f_{ij})$
satisfies the following properties:

1. Each column of the matrix $F$ is a $l^{2}-$sequence;

2. $F$ has an unique left inverse matrix $G^{*}=(g_{kl})$ such that each row of $G^{*}$ is also a $l^{2}-$sequence;

3. Operators $Q_{k}$ defined by the matrix $Q_{k}=FP_{k}G^{*}$ are well-defined projections on $\mathcal{H}$
and converges to the unit operator $I$ in the strong operator
topology.

Then the sequence $\{f_{k}\}_{k=1}^{\infty}, f_{k}=\sum_{i=1}^{\infty} f_{ik}e_{i}$ must be a Schauder basis.
\end{lemma}
Here we use the term ``left reverse'' in the classical means, that is, the series $\sum_{j=1}^{\infty}g_{kj}f_{jn}$ converges
absolutely to $\delta_{kn}$ for $k, n=1, 2, \cdots$. $G^{*}$ does not mean the adjoint of $G$, it is just a notation.
\begin{proof}
Property 1 just ensure that series $\{f_{k}=\sum_{j=1}^{\infty} f_{ij}e_{i}\}_{k=0}^{\infty}$ converges to a well-defined
vector $f_{k}$ in $\mathcal{H}$ by norm.
Property 2 implies that $span \{f_{n}; n=1, 2, \cdots\}=\mathcal{H}$ by the uniqueness of the left inverse. Moreover,
the $k-$th row of the matrix $G^{*}$ is just the vector $g_{k}^{*}$ such that
$(g_{k}^{*}, f_{n})=\delta_{kn}$. Therefore the vector sequence $\{f_{n}\}_{n=1}^{\infty}$ must be minimal
by the Hahn-Banach theorem(cf, \cite{Conway} corollary6.8, p82) and the Riesz representation theorem(see, \cite{Conway}, theorem3.4, p12).

Now for each vector $x=(x_{1}, x_{2},\cdots)$ denote by
$
\alpha_{k}^{x}= (g_{k}^{*}, x),
$
it is easy to check that $Q_{k}^{2}=Q_{k}$ and
$$
Q_{k}x=FP_{k}G^{*}x=\sum_{n=1}^{k}\alpha^{x}_{k} f_{k}.
$$
By property 3, we have
$Q_{k}x \rightarrow x$ since $Q_{k}$ converges to $I$ in strong operator topology(SOT). That is, series
$\sum_{n=1}^{\infty} \alpha^{x}_{n}f_{n}$ converges to the vector $x$ in norm. So we have proved that each
vector $x$ in $\mathcal{H}$ can be represented by the sequence $\{f_{n}\}_{n=1}^{\infty}$
with coefficients $\{\alpha_{n}^{x}\}_{n=1}^{\infty}$.

To show that
$\{f_{n}\}_{n=1}^{\infty}$ is a Schauder basis, we just need to show that this representation is unique.
Suppose that $\{\alpha_{n}\}_{n=1}^{\infty}$ is a sequence such that the series $\sum_{n=1}^{\infty} \alpha_{n}f_{n}$
converges to $0$ in the norm of the Hilbert space $\mathcal{H}$. Assume that the integer $n_{0}$ is the first number
satisfying $\alpha_{n_{0}} \ne 0$. Then we have
$$
f_{n_{0}}=-\frac{1}{\alpha_{n_{0}}} \cdot \sum_{n=n_{0}+1}^{\infty} \alpha_{n}f_{n}
$$
in which the series also converges in the norm topology. It counter to the fact that the sequence $\{f_{n}\}_{n=1}^{\infty}$ is a minimal sequence.
\end{proof}

Conversely, suppose that $\psi=\{f_{n}\}_{n=1}^{\infty}$ is a basis of $\mathcal{H}$. For a fixed ONB $\{e_{n}\}_{n=1}^{\infty}$,
each vector $f_{n}$ has a representation $f_{n}=\sum_{k=1}^{\infty} f_{kn}e_{k}$. Denote  $F_{\psi}=(f_{kn})$.
We shall call $F_{\psi}$ \textsl{the Schauder matrix corresponding to the basis $\psi$}. The following lemma is the inverse of
the above lemma.

\begin{lemma}\label{Lemma:Matrix Form 2}
Assume that $\psi=\{f_{n}\}_{n=1}^{\infty}$ is a Schauder basis. Then the corresponding Schauder matrix $F_{\psi}$ satisfies the following
properties:

1. Each column of the matrix $F_{\psi}$ is a $l^{2}-$sequence;

2. $F_{\psi}$ has an unique left inverse matrix $G_{\psi}^{*}=(g_{kl})$ such that each row of $G_{\psi}^{*}$ is also a $l^{2}-$sequence;

3. Operators $Q_{k}$ defined by the matrix $Q_{k}=F_{\psi}P_{k}G_{\psi}^{*}$ are well-defined projections on $\mathcal{H}$
and converges to the unit operator $I$ in the strong operator
topology.
\end{lemma}
\begin{proof}
Property 1 comes from the fact that $f_{n}$ is a vector in $\mathcal{H}$.

If $\{f_{k}\}_{k=1}^{\infty}$ is a Schauder basis, then the subspace $\widehat{\mathcal{H}}_{k}=span\{f_{n}; n\ne k\}$ for each $k$
satisfying $f_{k} \notin \widehat{\mathcal{H}}_{k}$(cf, \cite{Singer}, p50-51).
So we must have a unique linear functional $\varphi_{k}$ such that
$\varphi_{k}(f_{n})=\delta_{kn}$.
Then by the Riesz representation theorem, there is a unique vector $g^{*}_{k}=(g^{*}_{kl}) \in \mathcal{H}$ such that
$\sum_{j=1}^{n}g^{*}_{kj}, f_{jn}=\delta_{kn}$ in which $\{g^{*}_{kl}\}_{l=1}^{\infty}$ is a $l^{2}-$sequence. The uniqueness holds
because the sequence $\{f_{k}\}_{k=1}^{\infty}$ spans the Hilbert space.
Hence a Schauder matrix must have a unique
left inverse matrix whose rows are $l^{2}-$sequence. Then we have proved the property 2.

Property 3 is just a direct corollary
of the definition of Schauder basis.
Denote by $G=(G^{*})^{*}=g_{nk}$ the adjoint matrix of $G^{*}$, then we have $g_{nk}=\overline{g_{kn}}$. Moreover,
denote by $g_{n}$ the $n-$th column vector
and for a vector $x=\sum_{n=1}^{\infty} x_{n}e_{n}$ denote by
$y_{n}=\sum_{k=1}^{\infty} g_{nk}^{*}x_{k}$. Then trivially we have $y_{n}=(x, g_{n})$ and $(f_{k}, g_{n})=\delta_{kn}$
Suppose that $x=\sum_{k=1}^{\infty} \alpha_{k}f_{k}$ is the representation
of the vector $x$ under the basis $\psi$.
Then we must have $\alpha_{n}=y_{n}$ since
$$
y_{n}=(x, g_{n})=(\sum_{k=1}^{\infty} \alpha_{k}f_{k}, g_{n})=\alpha_{n}.
$$
Therefore we have $Q_{k}x=\sum_{n=1}^{\infty} \alpha_{n}f_{n}$. Clearly we have $Q_{k}x \rightarrow x$ in the norm topology.
In other words, $||Q_{k}x-x|| \rightarrow 0$ when $k\rightarrow \infty$
which implies $Q_{k} \rightarrow I$ in SOT(cf, \cite{Conway}, proposition 1.3, p262).
\end{proof}

The matrix $G_{\psi}^{*}$ is unique and decided completely by $F_{\psi}$. 
In fact the matrix $G^{*}$ is also the ``right inverse'' of the matrix $F$ in the classical sense. For more details, let
$F=(f_{kn})_{\omega \times \omega}$, $G^{*}=(g_{mk})_{\omega \times \omega}$, $f_{n}=\{f_{kn}\}_{k=1}^{\infty}$ and
$g_{m}^{*}=\{g_{mk}^{*}\}_{k=1}^{\infty}$. Moreover, denote their adjoint matrices by
$F^{*}=(f^{*}_{kn})_{\omega \times \omega}=(\overline{f_{nk}})_{\omega \times \omega}$,
$G=(g^{*}_{mk})_{\omega \times \omega}=(\overline{g_{km})}_{\omega \times \omega}$.
Then both $\psi=\{f_{n}\}_{n=1}^{\infty}$ and $\psi^{*}=\{g_{m}\}_{m=1}^{\infty}$
are biorthogonal basis to each other. That is, $\psi$ and $\psi^{*}$ are bases and we have $(f_{n}, g_{m})=\delta_{nm}$ for
all $n, m \in \mathbb{N}$. Now we show that the series
$
\sum_{k=1}^{\infty} f_{nk}g^{*}_{km}
$
converges to $\delta_{nm}$ as $k \rightarrow \infty$ for all $n, m \in \mathbb{N}$. Let $\{e_{l}\}_{l=1}^{\infty}$ be the
corresponding ONB. We write $e_{n}, e_{m}$ into the linearly combinations of basis vector in $\psi$ and $\psi^{'}$ as follows:
$$
e_{n}=\sum_{k=1}^{\infty}\alpha_{nk}f_{k}, e_{m}=\sum_{k=1}^{\infty}\beta_{mk}g^{*}_{k}.
$$
Then we have $\alpha_{nk}=g^{*}_{kn}$ and $\beta_{mk}=f^{*}_{km}=\overline{f_{mk}}$. Hence for any integer $N$
$$
\begin{array}{rl}
\sum_{k=1}^{N}f_{nk}g^{*}_{km}&=(\sum_{k=1}^{N}\alpha_{nk}f_{k}, \sum_{k=1}^{N}\beta_{mk}g^{*}_{k})\\
                              &=(e_{n}-\sum_{k=N}^{\infty}\alpha_{nk}f_{k}, e_{m}-\sum_{k=N}^{\infty}\beta_{mk}g^{*}_{k}).
\end{array}
$$
Now given $\epsilon >0$, we choose an integer $N$ such that inequalities
$||e_{n}-\sum_{k=N}^{\infty}\alpha_{nk}f_{k}||<\frac{\epsilon}{2}, ||e_{m}-\sum_{k=N}^{\infty}\beta_{mk}g^{*}_{k}||<\frac{\epsilon}{2}$ hold. Then we have
$$
\begin{array}{rl}
&|\sum_{k=1}^{N}f_{nk}g^{*}_{km}-(e_{n}, e_{m})| \\
=&|-(\sum_{k=N}^{\infty}\alpha_{nk}f_{k}, e_{m}-\sum_{k=N}^{\infty}\beta_{mk}g^{*}_{k})\\
                              &~~~~-(e_{n}-\sum_{k=N}^{\infty}\alpha_{nk}f_{k}, \sum_{k=N}^{\infty}\beta_{mk}g^{*}_{k})
                               +(\sum_{k=N}^{\infty}\alpha_{nk}f_{k}, \sum_{k=N}^{\infty}\beta_{mk}g^{*}_{k})|\\
                              \le & \epsilon(|1+\frac{\epsilon}{2}|+\frac{\epsilon}{4}).
\end{array}
$$
For this reason,  we have the following definition.
\begin{definition}
For a Schauder matrix $F_{\psi}$, the corresponding matrix $G_{\psi}^{*}$ is called the \textsl{inverse matrix} of $F_{\psi}$.
\end{definition}

If we do not ask that each row of $G^{*}$ is a $l^{2}-$sequence, an $\omega \times \omega$ matrix may have a ``left inverse''
in the classical sense.

\begin{example}
Let $F$ be the matrix
$$
\begin{bmatrix}
1&1&0&0&\cdots \\
0&-1&1&0&\cdots \\
0&0&-1&1&\cdots \\
0&0&0&-1&\cdots \\
\vdots&\vdots&\vdots&\vdots&\ddots
\end{bmatrix},
$$
and $G^{*}$ be the matrix
$$
\begin{bmatrix}
1&1&1&1&\cdots \\
0&-1&-1&-1&\cdots \\
0&0&-1&-1&\cdots \\
0&0&0&-1&\cdots \\
\vdots&\vdots&\vdots&\vdots&\ddots
\end{bmatrix}.
$$
It is trivial to check $G^{*}F=FG^{*}=I$. Then by above lemma \ref{Matrix Form} we know $F$ is not a Schauder matrix
since the rows of its inverse matrix are not $l^{2}-$sequence.
Moreover, if we denote by $g_{n}$ the $n-th$ column vector, then the sequence
$\xi=\{g_{n}\}_{n=1}^{\infty}$ is a complete minimal sequence(see \cite{Singer}, p24 and p50 for definitions). It is easy to
check that $\xi$ is complete since the $l^{2}$-sequence $h_{n}=\{h_{n}(j)\}_{j=1}^{\infty}, h_{n}(j)=\delta_{nj}$ is in its range;
On the other hand, the row vector sequence $\{f_{k}\}_{k=1}^{\infty}$ satisfies $(g_{n}, f_{k})=\delta_{kn}$ which implies
$g_{n} \notin \vee_{m \ne n} g_{m}$(or in notations of singer, we have $g_{n}\notin [g_{1}, \cdots, g_{n-1}, g_{n+1}, \cdots]$)
by the fact $\vee_{m \ne n} g_{m}=\ker \varphi_{k}$ in which $\varphi_{k}(x)=(x, f_{k})$ is a bounded functional by Riesz's theorem.
Therefore $\xi$ is an example which is complete and minimal sequence but not a basis sequence.
\end{example}

By above lemma \ref{Matrix Form} and \ref{Lemma:Matrix Form 2}, we have
\begin{theorem}\label{Theorem: Schauder Matrix}
An $\omega \times \omega$ matrix $F$ is a Schauder matrix if and only if it satisfies property 1, 2 and 3.
\end{theorem}

For a Schauder matrix $F$, the column vector sequence $\{g_{n}\}_{n=1}^{\infty}$ of $G$ defined
in above lemmas is also a Schauder basis which is called the
\textsl{biorthogonal basis} to the basis $\{f_{k}\}_{k=1}^{\infty}$(cf \cite{Y2}, pp23-29, \cite{Singer} pp23-25).

The projection $FP_{n}G^{*}$ is just the $n-$th ``natural projection'' so called in \cite{B}(p354).
It is also the \text{$n-$th partial sum operator} so called in \cite{Singer}(definition 4.4, p25). Now we can translate
theorem 4.1.15 and corollary 4.1.17 in \cite{B} into the following
\begin{proposition}\label{BC1}
If $F$ is a Schauder matrix, then $M=\sup_{n} \{||FP_{n}G^{*}||\}$ is a finite const.
\end{proposition}

The const $M$ is called the \textsl{basis const} for the basis $\{f_{n}\}_{n=1}^{\infty}$.

Assume that $\psi=\{f_{n}\}_{n=1}^{\infty}$ is a basis.
For a subset $\Delta$ of $\mathbb{N}$, denote by $P_{\Delta}$ the diagonal matrix defined as
$P_{\Delta}(nn)=1$ for $n \in \Delta$ and $P_{\Delta}(nn)=0$ for $n \notin \Delta$.
The projection $Q_{\Delta}=F_{\psi}P_{\Delta}G_{\psi}^{*}$ defined in above lemmas
is called a \textsl{natural projection}(see, definition 4.2.24, \cite{B}, p378).
In fact for a vector $x=\sum_{n=1}^{\infty} x_{n}f_{n}$, it is trivial to check $Q_{\Delta} x=\sum_{n \in \Delta} x_{n}f_{n}$.
Then we have a same result for the \textsl{ unconditional basis const}(cf, definition4.2.28, \cite{B}, p379):
\begin{proposition}\label{UBC1}
If $F_{\psi}$ is a Schauder matrix, then the unconditional basis
const of the basis $\psi$ is $M_{ub}=\sup_{\Delta \subseteq \mathbb{N}} \{||F_{\psi}P_{\Delta}G_{\psi}^{*}||\}$.
\end{proposition}

In virtue of the proposition 4.2.29 and theorem 4.2.32 in the book \cite{B}, we have
\begin{proposition}\label{Proposition: Unconditional Matrix}
For a Schauder basis $\psi$, it is an unconditional basis
if and only if $\sup_{\Delta \subseteq \mathbb{N}} \{||F_{\psi}P_{\Delta}G_{\psi}^{*}||\}<\infty$.
\end{proposition}

Following the notations in lemma \ref{Matrix Form}, as a direct corollary of lemma \ref{Matrix Form} and theorem 6 in \cite{Y2}(p28), we have

\begin{proposition}
$F$ is a Schauder matrix if and only if the adjoint matrix (conjugate transpose) $G$ of its left inverse $G^{*}$ is a Schauder matrix.
\end{proposition}

As well known that a sequence of operators $T_{n}$ converges to an operator $T$ in SOT dose not imply $T_{n}^{*}$ converging
to $T$ in SOT, so the above proposition is not trivial.

\begin{corollary}
$M=\sup_{n} \{||FP_{n}G^{*}||\}<\infty$ if and only if $M^{'}=\sup_{n} \{||GP_{n}F^{*}||\}$ $<\infty$.
\end{corollary}

\subsection{}From the definition of the Schauder matrix $F_{\psi}$, basic properties of Schauder matrix have natural relations to the Schauder basis $\psi$. This understanding lead us to the following definition.

\begin{definition}
A matrix $F$ is called an \textsl{unconditional, conditional, Riesz, normalized} or \textsl{quasinormal} respectively if and only if
the sequence of its column vectors comprise an unconditional, conditional, Riesz, normalized or quasinormal basis.  Two Schauder
matrices $F_{\psi}, F_{\varphi}$ are called \textsl{equivalent} if and only if the corresponding bases $\psi$ and $\varphi$ are equivalent.
\end{definition}

Here we use the term quasinormal instead of ``bounded'' to avoid ambiguity(cf \cite{Ole} p476, \cite{Singer} p21).
Arsove use the word ``similar'' in the same meaning as the word ``equivalent''(cf, \cite{Arsove} p19, \cite{B}p387).

Denote by $\pi_{\infty}$ the set of all permutations of $\mathbb{N}$(see \cite{Singer}, p361).
Denote by $U_{\pi}$ both the unitary operator which maps $e_{\pi(n)}$ to $e_{n}$ and the corresponding matrix under the
ONB $\{e_{n}\}_{n=1}^{\infty}$.

\begin{theorem}\label{BPS}
Assume that $F$ is a Schauder matrix and $G^{*}$ is its inverse matrix. We have \\
1. For each invertible matrix $X$, $XF$ is also a Schauder matrix. Moreover, $XF$ is unconditional(conditional)
if and only if $F$ is unconditional(conditional);\\
2. For each diagonal matrix $D=diag(\alpha_{1}, \alpha_{2}, \cdots)$ in which each diagonal element $\alpha_{k}$ is nonzero,
$FD$ is also a Schauder matrix. Moreover, $FD$ is unconditional(conditional)
if and only if $F$ is unconditional(conditional); \\
3. For a unconditional matrix $F$, $FU$ is also a unconditional matrix for $U \in \pi_{\infty}$; \\
4. Two Schauder matrix $F$ and $F^{'}$ are equivalent if and only if there is a invertible matrix $X$ such
that $XF=F^{'}$.
\end{theorem}
\begin{proof}
Property 1, 2, 3 and 4 are basic facts about basis just in a matrix language. Their counterparts are
proposition 4.1.8, 4.2.14, 4.1.5, 4.2.12, and corollary 4.2.34 in \cite{B}, Theorem 1 in \cite{Arsove}.
Some of those facts are easy to check by our
lemma \ref{Matrix Form}. As an example, we shall prove property 1. Let $F^{'}=XF$, then clearly  $G^{*'}=GX^{-1}$ is its inverse matrix.
Both properties 1 and 2 in lemma \ref{Matrix Form} hold immediately. To verify property 3, we know that
$FP_{n}G^{*}$ converges to $I$ in SOT if and only if $XFP_{n}G^{*}X^{-1}$ converges to $I$ in SOT. Also we have
$$
||XFPG^{*}X^{-1}|| \le ||X||\cdot||X^{-1}|| \cdot ||FPG^{*}||
$$
for any natural projection $P$, which implies the last part of property 1(cf, \cite{B} theorem 4.2.32).
\end{proof}

\subsection{}Now we turn to study the basic properties of Schauder operators.
Recall that a Schauder operator $T$ is an operator mapping
some ONB into a Schauder basis.
In his paper \cite{Ole}, Olevskii call an operator to be \textsl{generating} if and only if it maps some ONB
into a quasinormal conditional basis. Hence our definition of Schauder operator is a generalization of Olevskii's
one.

\begin{theorem}\label{BPSO}
Following conditions are equivalent:\\
1. $T$ is a Schauder operator;\\
2. $T$ maps some ONB $\{e_{n}\}_{n=1}^{\infty}$ into a basis;\\
3. $T$ has a polar decomposition $T=UA$ in which $A$ is a Schauder operator;\\
4. Assume that $T$ has a matrix representation $F$ under a fixed ONB $\{e_{n}\}_{n=1}^{\infty}$. There
is some unitary matrix $U$ such that $FU$ is a Schauder matrix.
\end{theorem}
\begin{proof}
$2\Rightarrow 1$. The $k-$th column of the matrix of $T$ under the ONB $\{e_{n}\}_{n=1}^{\infty}$ is just the
$l^{2}-$coefficients of $Te_{k}$.

$1\Rightarrow 3$. Assume that $\{f_{n}\}_{n=1}^{\infty}$ is a basis in which $f_{n}$ is the $n-$th column
of the matrix $F$ of $T$ under some ONB. Then if we denote the matrix of $U$ and $A$ also by the same notations,
we have $UA=F$. Property 1 of lemma \ref{BPS} tell us $U^{*}F=A$ is also a Schauder matrix.

$3\Rightarrow 4$. Assume that $\{g_{n}\}_{n=1}^{\infty}$ is an ONB such that the matrix of $A$ under it is a Schauder matrix.
Then the operator $U$ defined as $Ue_{n}=g_{n}$ is a unitary operator and the $n-$th column of its matrix under the
ONB $\{e_{n}\}_{n=1}^{\infty}$ is just the $l^{2}-$coefficients of $g_{n}$. Hence we have $AU$ is a Schauder matrix.

$4\Rightarrow 1$. The column vector sequence of the unitary matrix $U$ is an ONB. The matrix of $T$ under this ONB
is just $U^{*}FU$. Property 1 of lemma \ref{BPS} shows that $U^{*}FU$ is a Schauder matrix since $FU$ is a Schauder
matrix itself.
\end{proof}

The equivalence $1 \Leftrightarrow 3$ had been used in proof of the theorem $1^{'}$ of \cite{Ole}, although
Olevskii had not given an explanation.

\begin{proposition}\label{IDR}
A Schauder operator $T$ must be injective and has a dense range in $\mathcal{H}$.
\end{proposition}
\begin{proof}
$T$ must be injective since the representation of $0$ is unique. For a basis $\{f_{n}\}_{n=1}^{\infty}$,
the finite linear combination of $\{f_{n}\}_{n=1}^{\infty}$ is dense in the Hilbert space $\mathcal{H}$.
Therefore the range of $T$ must be dense in $\mathcal{H}$.
\end{proof}


\subsection{}If $T$ is a Schauder operator, does for each ONB sequence $\{e_{n}\}_{n=1}^{\infty}$ the vector sequence
$\{Te_{n}\}_{n=1}^{\infty}$ always be a basis? In this subsection, we shall show that
the answer is negative in general and it is true only in the case that $T$
is an invertible operator.

\begin{lemma}\label{LP}
Assume that $A$ is a positive operator satisfying $\sigma(A) \subseteq [\lambda_{1}, \lambda_{2}]$
and $\lambda_{1}, \lambda_{2}\in \sigma(A)$ for some $\lambda_{1}>0$. Then for any const $\varepsilon>0$
small enough,
there is a rank 1 projection $P$ such that
$\frac{1}{2\sqrt{2}}\frac{\lambda_{2}}{\lambda_{1}}-\varepsilon<||APA^{-1}||$.
\end{lemma}
\begin{proof}
Let $e_{1}, e_{2}$ be two normalized vectors in $\mathcal{H}$ such that
$$
e_{1} \in E_{[\lambda_{1}, \lambda_{1}+\delta]}, e_{2} \in E_{[\lambda_{2}-\delta, \lambda_{2}]}.
$$
in which $E_{[\lambda_{1}, \lambda_{1}+\delta]}$ and $E_{[\lambda_{2}-\delta, \lambda_{2}]}$ is the spectral projection of $A$
on the interval $[\lambda_{1}, \lambda_{1}+\delta]$ and $[\lambda_{2}-\delta, \lambda_{2}]$ respectively(cf, \cite{Conway}, pp269-272).
Then for $\delta < \frac{\lambda_{2}-\lambda_{1}}{2}$, we have $(e_{1}, e_{2})=0$ and
$$
\lambda_{1} \le||Ae_{1}|| \le \lambda_{1}+\delta, \lambda_{2}-\delta \le ||Ae_{2}|| \le \lambda_{2}.
$$
Consider the vector $e=\frac{1}{\sqrt{2}}e_{1}+\frac{1}{\sqrt{2}}e_{2}$
and the operator $P=e\otimes e$ defined as:
$
Px=(x, e)e.
$
It is trivial to check that $P$ is a rank 1 orthogonal projection. Now we have
$
APA^{-1}(x)=(A^{-1}x, e)Ae,
$
hence $||APA^{-1}||=\sup_{||x||=1} ||APA^{-1}x||$. Then we have
$$
\begin{array}{rl}
(A^{-1}e, e)=&\frac{1}{\sqrt{2}}(A^{-1}e_{1}, e)+\frac{1}{\sqrt{2}}(A^{-1}e_{2}, e) \\
            =&\frac{1}{2}\{(A^{-1}e_{1}, e_{1})+(A^{-1}e_{2}, e_{2})\} \\
         \ge &\frac{1}{2}\{\frac{1}{\lambda_{1}+\delta}+\frac{1}{\lambda_{2}}\}
\end{array}
$$
and
$$
||Ae||^{2} \ge \frac{1}{2}\lambda_{1}^{2}+\frac{1}{2}(\lambda_{2}-\delta)^{2}.
$$
Therefore the following inequality holds:
$$
\begin{array}{rl}
||APA^{-1}e|| \ge & \frac{1}{2}\{\frac{1}{\lambda_{1}+\delta}+\frac{1}{\lambda_{2}}\}\sqrt{\frac{1}{2}\lambda_{1}^{2}+\frac{1}{2}(\lambda_{2}-\delta)^{2}} \\
\ge & \frac{1}{2\sqrt{2}}\frac{\lambda_{2}-\delta}{\lambda_{1}+\delta}.
\end{array}
$$
Let $\varepsilon$ be a const satisfying $\varepsilon < \frac{1}{2\sqrt{2}}$.
Hence for the positive number $\delta <\frac{2\sqrt{2}\lambda_{1}^{2}\varepsilon}{(1-2\sqrt{2}\varepsilon)\lambda_{1}+\lambda_{2}}$
the required inequality holds.
\end{proof}

\begin{theorem}
If an operator $A$ maps every ONB sequence into a basis, then $A$ must be an invertible operator.
\end{theorem}
\begin{proof}A direct result of \ref{BPS} is that if an operator $A$ maps every ONB into a basis then it maps each ONB into a unconditional
basis. By virtue of theorem \ref{BPSO}, we can assume that $T$ is a positive operator. We need to show that $0 \notin \sigma(A)$.
Firstly, we have $0 \notin \sigma_{p}(A)$ by above proposition \ref{IDR} since $A$ is a Schauder operator. If
$0 \in \sigma(p)$ then $0$ must be an accumulation point of $\sigma(T)$. Hence we can choose a sequence $\{\lambda_{k}\}_{k=1}^{\infty}$
such that: \\
1. $\{\lambda_{k}\}_{k=1}^{\infty} \subseteq \sigma(A)$; and \\
2. $\lambda_{k+1}<\lambda_{k}$ and $\frac{\lambda_{2n}}{\lambda_{2n-1}}<\frac{1}{n+1}$.\\
Denote by $I_{0}=\sigma(A)-\cup_{n=1}^{\infty}[\lambda_{2n}, \lambda_{2n-1}]$ and $A_{0}=AE_{I_{0}}$.
Let $A_{n}=AE_{[\lambda_{2n}, \lambda_{2n-1}]}$, then we have $A=A_{0}\oplus A_{1}\oplus A_{2}\oplus A_{3} \cdots$.
And each operator $A_{n}$ is an
invertible positive operator for $n \ge 1$. Now by above lemma \ref{LP}, we can choose a vector
$e_{1}^{(n)} \in Ran E_{[\lambda_{2n}, \lambda_{2n-1}]}$ such that the
projection $P^{(n)}_{1}=e_{1}^{(n)}\otimes e_{1}^{(n)}$ satisfying
$$
AP^{(n)}_{1}A^{-1}=A_{n}P^{(n)}_{1}A_{n}^{-1}>n
$$
for each n. Here we use the fact
$$
E_{[\lambda_{2n}, \lambda_{2n-1}]}P^{(n)}_{1}=P^{(n)}_{1}E_{[\lambda_{2n}, \lambda_{2n-1}]}=P^{(n)}_{1}.
$$
Now for each subspace $Ran E_{[\lambda_{2n}, \lambda_{2n-1}]}$ we choose an ONB $\{f_{k}^{(n)}\}_{k=1}^{\alpha_{k}}$ such that
$e^{(n)}_{1}=f^{(n)}_{1}$. Moreover, choose an ONB $\{e^{(0)}_{k}\}_{k=1}^{\alpha_{0}}$ of the subspace $Ran E_{I_{0}}$.
Here $\alpha_{k}$ is a finite number or the countable cardinal which is equal to the dimension of the subspace
$Ran E_{[\lambda_{2n}, \lambda_{2n-1}]}$ and $Ran E_{I_{0}}$ respectively.
Clearly the set
$\{f^{(n)}_{k}; n=0, 1, 2,\cdots \hbox{ and }k=1, 2, \cdots, \alpha_{k}\}$ is an ONB for $\mathcal{H}$ itself. It is a
countable set and each its arrangement $\psi$ give an ONB sequence of $\mathcal{H}$. In more details, denote by
$\Delta =\{(n, k); n=0, 1, 2, \cdots \hbox{ and } k=1, 2, \cdots, \alpha_{k}\}$.
For any bijection
$\sigma : \Delta \rightarrow \mathbb{N}$, define $g_{n}=f_{t}^{(s)}, (s, t)=\sigma^{-1}(n)$. Then $\psi_{\sigma}=\{g_{n}\}_{n=1}^{\infty}$
is an ONB sequence.
\begin{claim}
For each ONB sequence $\psi_{\sigma}$, $\{Ag_{n}\}_{n=1}^{\infty}$ is not a basis.
\end{claim}
We have shown that if $\{Ag_{n}\}_{n=1}^{\infty}$ is a basis it must be a unconditional one. So it is enough to show
that it is not a unconditional basis, which can be verified by its unconditional const. Assume that the claim is not true, that is,
$\{Ag_{n}\}_{n=1}^{\infty}$ is a basis.
It is trivial to check that
$A_{n}P^{(n)}_{1}A_{n}^{-1}$ is a natural projection corresponding to the basis  $\{Ag_{n}\}$. In fact, we have
$$
A_{n}P^{(n)}_{1}A_{n}^{-1}=P_{\sigma(n, 1)}-P_{\sigma(n, 1)-1}.
$$
Here we denote by $P_{n}$ the $n-th$ partial sum operator so called in the book \cite{Singer}.
But now we have $||A_{n}P^{(n)}_{1}A_{n}^{-1}||\rightarrow \infty$ which counters to the fact that  a unconditional basis
must have a finite unconditional const
(cf, \cite{B}, corollary4.2.26).
\end{proof}

\begin{corollary}
If an operator $T$ is not invertible, then there is some ONB $\{e_{n}\}_{n=1}^{\infty}$ such that the
sequence  $\{Te_{n}\}_{n=1}^{\infty}$ is not a basis.
\end{corollary}

By the theorem 1 of \cite{Ole}, a generating operator never be invertible. Hence we have
\begin{corollary}
For a generating operator $T$, there is some ONB $\{e_{n}\}_{n=1}^{\infty}$ such that the
sequence  $\{Te_{n}\}_{n=1}^{\infty}$ is not a basis.
\end{corollary}

Both the English translation and the review(MR0318848) of the paper \cite{Ole} by A. M. Oleskii  make a pity clerical mistake:

Review(MR0318848):``The author obtains a spectral characterization for the linear operators that transform $\mathbf{every}$ complete
orthonormal system into a conditional basis in a Hilbert space.''

The English translation: ``Definition. A bounded noninvertible linear operator $T: \mathcal{H} \rightarrow \mathcal{H}$ is said to be generating if it maps
$\mathbf{every}$ orthonormal basis $\varphi$ into a quasinormed basis $\psi$.''

The word ``every'' should be ``some'' in both of them.
Note that in the proof of the theorem 1 (\cite{Ole}), Olevskii had shown that an operator never can maps every ONB into
a conditional basis. Even the theorem 1 of \cite{Ole} itself shows it, but need a little operator theory discussion.

Since in the Hilbert space $\mathcal{H}$ all quasinormal unconditional bases are equivalent(cf, Theorem 18.1, \cite{Singer}, p529)
 and in addition with theorem \ref{BPS}, we have
\begin{proposition}\label{RM}
An $\omega\times \omega$ matrix $F$ is a Riesz matrix if and only if it represents an invertible operator.
\end{proposition}
Above result also can be obtained directly form theorem 2 of the paper \cite{J-M}.

\begin{corollary}
An operator $T$ is invertible if and only if there is some ONB such that the matrix $F$ under this ONB of $T$
is a Riesz matrix.
\end{corollary}

\begin{corollary}\label{Corollary: Matrix of Invertible Op}
For an invertible operator $T$, its matrix always be a Riesz matrix under any ONB.
\end{corollary}


\subsection{}Conditional and unconditional bases have very different behaviors. On the other side,
properties of operators given by Schauder matrices are strongly dependent on the related bases. Both the theorem 1 of
the paper \cite{Ole} and the behaviors of Riesz matrix(cf, proposition \ref{RM}) support this observation.
In this subsection, we give a same
classification of operators dependent on their matrix representation(Or equivalently, on their actions on ONBs).
And then we give some more remarks on Olevskii's paper.
\begin{definition}\label{CUOP}
A Schauder operator $T$ will be called a conditional operator if and only if there is some ONB $\{e_{n}\}_{n=1}^{\infty}$
such the column vector sequence of its matrix representation $F$ of $T$ under the ONB comprise a conditional basis. Otherwise,
$T$ will be called a unconditional operator.
\end{definition}

By the theorem \ref{BPSO}, we have
\begin{corollary}
A Schauder operator $T$ is conditional if and only if it maps some ONB $\{e_{n}\}_{n=1}^{\infty}$ into a conditional basis
$\{Te_{n}\}_{n=1}^{\infty}$.
\end{corollary}

For convenience, we correct the error appearing in the translation and rewrite Olevskii's definition as follows:
\begin{definition}
A bounded operator $T \in \mathcal{L}(\mathcal{H})$ is said to be generating if and only if it maps some ONB into a quasinormal conditional
basis.
\end{definition}
Above definition modifies slightly from the original form on the Olevskii's paper. We write down the original one
to compare them in details:
\begin{definition}(\cite{Ole}, p476)
A bounded non-invertible operator $T: \mathcal{H} \rightarrow \mathcal{H}$ is said to be generating if and only if it maps
some ONB into a quasinormal basis.
\end{definition}

\begin{proposition}
Above two definitions are equivalent.
\end{proposition}
\begin{proof}
If a bounded operator $T \in \mathcal{L}(\mathcal{H})$ maps some ONB into a quasinormal conditional
basis, then it must be non-invertible since an invertible operator maps each ONB into a Riesz basis(hence a unconditional basis) by
proposition \ref{RM};
On the other side, If a bounded non-invertible operator $T: \mathcal{H} \rightarrow \mathcal{H}$  maps
some ONB into a quasinormal basis. Then the quasinormal basis must be a conditional one otherwise $T$ must be invertible again by
proposition \ref{RM}.
\end{proof}

\begin{corollary}
A generating operator is a conditional operator; An invertible operator is a unconditional operator.
\end{corollary}
\section{ A Criterion for Operators to be Conditional}
\subsection{}{\bf Question: \ } Is $K=diag\{1,\frac{1}{2},\frac{1}{3},\cdots\}$ a conditional operator?

From Olevskii's result, we can not obtain the confirm answer. In this section, we will improve the Olevskii's technology and gain a confirm answer.

First, let us recall some notations in the line of Olevskii.

Let $A_k=\left(
            \begin{array}{c}
              a_{ij} \\
            \end{array}
          \right)
\in M_{2^k}(\mathbb{C})$ (where $1\leq i,j\leq 2^k$) be defined as follows: $a_{i1}=2^{-\frac{k}{2}}, 1\leq i\leq 2^k$; and if $j=2^s+v(1\leq v\leq 2^s)$, then

\begin{align*}
a_{ij}&=\left\{
\begin{array}{ll}2^{\frac{s-k}{2}},\hspace{10mm}  (v-1)2^{k-s}< i\leq (2v-1)2^{k-s-1},
\\[2mm]
-2^{\frac{s-k}{2}}, \hspace{10mm}(2v-1)2^{k-s-1}< i\leq v2^{k-s}.\\
\end{array}\right.
\end{align*}

For $\alpha, \frac{1}{\sqrt{2}}<\alpha<1$, let $T_{(k,\alpha)}\in M_{2^k}(\mathbb{C})$ be defined as follows:
$$T_{(k,\alpha)}=\begin{bmatrix}\begin{bmatrix}                     \alpha^k &  \\
                                                                         & \alpha^k \\
                                                                     \end{bmatrix}
                                                                      \\
                                                                    & \begin{bmatrix}
                                                                          \alpha^{k-1}  \\
                                                                           & \alpha^{k-1} \\
                                                                        \end{bmatrix}  &  \\
                                                                       &  & \ddots &  \\
                                                                       &  &  & \begin{bmatrix}
                                                                                    \alpha & \\
                                                                                     & \ddots \\
                                                                                     &&\alpha
                                                                                  \end{bmatrix} _{2^{k-1}\times2^{k-1}}
                                                                     \\
                                                                 \end{bmatrix}.
$$

In this section, we will show that if the positive operator $T$ does not admit the eigenvalue zero and $\sigma(T)$ has a decreasing sequence $\{\lambda_n,n=1,2,\ldots\}$ which converges to zero and $$\lim\limits_{n\rightarrow\infty}\frac{\lambda_n}{\lambda_{n+1}}=1,$$ then $T$ must be a conditional operator.
Thus the compact operator $K=diag\{1,\frac{1}{2},\frac{1}{3}$, $\cdots\}$ is a conditional operator.

\subsection{} Now, we give a key lemma.

\begin{lemma}\label{keylemma}
Let $T$ be a diagonal operator with entries $\{
\lambda_1, \lambda_2, \lambda_3,\ldots\}$ under the ONB $\{e_{k}\}_{k=1}^{\infty}$, where $\lambda_n>0$. Given $\alpha$, $\frac{1}{\sqrt{2}}<\alpha<1$.
If for each $k\geq1$, there exist positive numbers $c_k\leq d_k$, such that

a) $sup_k\frac{d_k}{c_k}<\infty$,

b) there exists subset $\triangle_k=\{n^k_1,n^k_2,\cdots,n^k_{2^k}\}$ of $\mathbb{N}$ such that $c_k\leq\frac{\alpha^k}{\lambda_{n^k_{2^k-1}}},\frac{\alpha^k}{\lambda_{n^k_{2^k}}} \leq d_k$, and $c_k\leq\frac{\alpha^j}{\lambda_{n^k_i}}\leq d_k$ when $1\leq j\leq k-1,~2^k(1-\frac{1}{2^{j-1}})+1\leq i\leq2^k(1-\frac{1}{2^{j}}),
$

c) $max{\triangle_k} < min{\triangle_{k'}}$ when $k< k'$,

then $T$ is a conditional operator.
\end{lemma}
\begin{proof}
In this proof, we shall
identify the operators and the $\omega \times \omega$ matrix representation of the operators under ONB $\{e_{k}\}_{k=1}^{\infty}$.

We rearrange $n^k_1,n^k_2,\cdots,n^k_{2^k}$ into a increasing sequence and denote it by $m^k_1,m^k_2$, $\cdots,m^k_{2^k}$ ($m^k_1<m^k_2<\cdots<m^k_{2^k}$). Let $m^0_1=1$ and $\mathcal {H}_k=span\{e_{m^k_1},e_{m^k_1+1}$, $\ldots,e_{m^{k+1}_1-1}\}$ for $k\geq0$, then since $max{\triangle_k} < min{\triangle_{k'}}$ when $k< k'$, we know $\mathcal {H}_k\cap \mathcal {H}_{k'}=(0)$ when $k\neq k'$ and $\oplus_{k\geq0}\mathcal {H}_k=\mathcal {H}$. Moreover, $\{\lambda_{n^k_1},\lambda_{n^k_2},\ldots,\lambda_{n^k_{2^k}}\}\subseteq\{\lambda_{m^k_1}, \lambda_{m^k_1+1},\ldots,$ $\lambda_{m^{k+1}_1-1}\}$ for any $k\geq1$.

Let $T_k\in \mathcal {L}(\mathcal {H}_k)$ the k-th block of $T$ on $\mathcal {H}_k$, i.e. $$T_k=\begin{bmatrix}
                    \lambda_{m^k_1} &  &  &  \\
                     & \lambda_{m^k_1+1} &  &  \\
                     & & \ddots\\
                     &  & & \lambda_{m^{k+1}_1-1} \\
                  \end{bmatrix}\begin{matrix}e_{m^k_1}\\
                e_{m^k_1+1}\\
                \vdots\\
                e_{m^{k+1}_1-1}\end{matrix},$$ then $\oplus_{k\geq0}T_k=T$.
 Denote  $\widetilde{T}_0=T_0$. For $k\geq1$, let $$\widetilde{T}_k=\begin{bmatrix}
                    \lambda_{n^k_{2^k}} &  &  &  \\
                     & \lambda_{n^k_{2^{k}-1}} &  &  \\
                     & & \ddots\\
                     &  & & \lambda_{n^k_1}& \\
                     &  &  & & S_k \\
                  \end{bmatrix}\begin{matrix}e_{m^k_1}\\
                e_{m^k_1+1}\\
                \vdots\\
                e_{m^{k}_1+2^k-1}\\
                \widetilde{\mathcal{H}}_k\end{matrix},
$$  where $\widetilde{\mathcal{H}}_k=\bigvee\{e_{m^{k}_1+2^k},\ldots, e_{m^{k+1}_1-1}\}$ and $S_k$ is a diagonal operator with entries $\{\lambda_{m^k_1}, \lambda_{m^k_1+1},\ldots,$ $\lambda_{m^{k+1}_1-1}\}\backslash\{\lambda_{n^k_1},\lambda_{n^k_2}$, $\ldots,\lambda_{n^k_{2^k}}\}$.  It is easy to see that the entries of $\widetilde{T}_k$ are just a rearrangement
 of entries of $T_k$ for $k\geq1$.

We will prove  $\widetilde{T}\triangleq\oplus_{k\geq0}\widetilde{T}_k$  is a conditional operator and then show $T$ is a conditional operator.

Let $X_0=I\in\mathcal{L}(\mathcal {H}_0)$. For $k\geq1$, let $$X_k=\begin{bmatrix}
              c_k\cdot\begin{bmatrix}
             \frac{\lambda_{n^k_{2^k}}}{\alpha^k} &  &  \\
              &\frac{\lambda_{n^k_{2^k-1}}}{\alpha^k}\\
              && \ddots &  \\
              &&&\frac{\lambda_{n^k_i}}{\alpha^j}\\
              &&&&\ddots\\
              && && & \frac{\lambda_{n^k_1}}{\alpha} \\
           \end{bmatrix} &  \\
               & I \\
            \end{bmatrix}\in \mathcal{L}(\mathcal {H}_k),
$$ since $$Sup_k max\{c_k\frac{\lambda_{n^k_{2^k}}}{\alpha^k},\ldots,c_k\frac{\lambda_{n^k_1}}{\alpha},c_k^{-1}\frac{\alpha^k}{\lambda_{n^k_{2^k}}},\ldots,c_k^{-1}\frac{\alpha}{\lambda_{n^k_1}}\}\leq Sup_k max\{1,\frac{d_k}{c_k}\}<\infty,$$ we have $X\triangleq\oplus_{k\geq0}X_k$ is an invertible operator.

Moreover for $k\geq1$, $$\widetilde{T}_k=X_k\cdot\begin{bmatrix}
                            T_{(k,\alpha)}c_k^{-1} &  \\
                             & S_k \\
                          \end{bmatrix},
$$ so  $$\widetilde{T}=\oplus_{k\geq0}\widetilde{T}_k=X\cdot\oplus_{k\geq0}\begin{bmatrix}
                            T_{(k,\alpha)}c_k^{-1} &  \\
                             & S_k \\
                          \end{bmatrix},$$ where we denote $\begin{bmatrix}
                            T_{(k,\alpha)}c_k^{-1} &  \\
                             & S_k \\
                          \end{bmatrix}$ by $\widetilde{T}_0$ when $k=0$.

Let $$U=\oplus_{k\geq0}\begin{bmatrix}
                    A_k^* &  \\
                     & I \\
                  \end{bmatrix}
,$$ where we denote $\begin{bmatrix}
                    A_k^* &  \\
                     & I \\
                  \end{bmatrix}=I\in\mathcal{L}(\mathcal {H}_0)$ when $k=0$, then it is an unitary operator and
\begin{eqnarray*}\widetilde{T}U&=&X\cdot\oplus_{k\geq0}\begin{bmatrix}
                            T_{(k,\alpha)}c_k^{-1} &  \\
                             & S_k \\
                          \end{bmatrix}\cdot \oplus_{k\geq0}\begin{bmatrix}
                    A_k^* &  \\
                     & I \\
                  \end{bmatrix}\\
                          &=&X\cdot\oplus_{k\geq0}\begin{bmatrix}
                            T_{(k,\alpha)} A_k^* c_k^{-1} &  \\
                             & S_k \\
                          \end{bmatrix} \\
                          &=& X\cdot\oplus_{k\geq0}\begin{bmatrix}
                            T_{(k,\alpha)} A_k^*  &  \\
                             & S_k \\
                          \end{bmatrix}\cdot\oplus_{k\geq0}\begin{bmatrix}
                           c_k^{-1}I &  \\
                             & I \\
                          \end{bmatrix}.\end{eqnarray*}

To show $\widetilde{T}$ is conditional, from theorem \ref{BPS}, it suffices to  show that $$F\triangleq\oplus_{k\geq0}\begin{bmatrix}
                            T_{(k,\alpha)} A_k^*  &  \\
                             & S_k \\
                          \end{bmatrix}$$ is a conditional matrix.

We will deal with it by theorem \ref{Theorem: Schauder Matrix} and proposition \ref{Proposition: Unconditional Matrix}. First, one can easily see that $F$ has an unique left inverse matrix $$G^*=\oplus_{k\geq0}\begin{bmatrix}
                           A_k T_{(k,\alpha)}^{-1} &  \\
                             & S_k^{-1} \\
                          \end{bmatrix}$$ where each row  is  a $l^2-$ sequence.

Second,  $Q_n=FP_nG^*$ are obviously projections. Let \begin{eqnarray*}&&\Lambda_1=\{m^k_1,m^k_1+1,\ldots,m^k_1+2^k-1;~k\geq1\}\subseteq\mathbb{N},\\
&&\Lambda_2=\{m^k_1+2^k,m^k_1+2^k+1,\ldots,m^{k+1}_1-1;~k\geq1\}\subseteq\mathbb{N}.\end{eqnarray*}    For any $x\in\mathcal {H}$, we have $$x=\sum\limits_{j=1}^{\infty}x_je_j=\sum\limits_{j\in\Lambda_1}x_je_j+\sum\limits_{j\in\Lambda_2}x_je_j,$$ and \begin{eqnarray*}&FP_nG^*(x)\\
                        &=FP_nG^*(\sum\limits_{j\in\Lambda_1}x_je_j+\sum\limits_{j\in\Lambda_2}x_je_j)\\
                        &=(\oplus_{k\geq0} T_{(k,\alpha)} A_k^*)P^{(1)}_n(\oplus_{k\geq0} A_k T_{(k,\alpha)}^{-1})(\sum\limits_{j\in\Lambda_1}x_je_j)+P^{(2)}_n(\sum\limits_{j\in\Lambda_2}x_je_j),\end{eqnarray*} where $\oplus_{k\geq0} T_{(k,\alpha)} A_k^*$ and $P^{(1)}_n$ are the operators on $\mathcal {H}^{(1)}=\bigvee_{j\in\Lambda_1}\{e_j\}$, $P^{(1)}_n$ converges to $I$ in the strong operator topology; $P^{(2)}_n$ is the operator on $\mathcal {H}^{(2)}=\bigvee_{j\in\Lambda_2}\{e_j\}$ and also converges to $I$ in the strong operator topology.

It follows from the result of Olevskii that $\oplus_{k\geq0} T_{(k,\alpha)} A_k^*$ is quasinormal conditional matrix. Then from theorem \ref{Theorem: Schauder Matrix}, we have $$\lim\limits_{n\rightarrow\infty}(\oplus_{k\geq0}  T_{(k,\alpha)} A_k^*)P^{(1)}_n(\oplus_{k\geq0}  A_k T_{(k,\alpha)}^{-1})(\sum\limits_{j\in\Lambda_1}x_je_j)=\sum\limits_{j\in\Lambda_1}x_je_j.$$ Thus $FP_nG^*(x)$ converges to $x$ as $n\rightarrow\infty$ and $F$ is a Schauder matrix.

Moreover, since the unconditional basis const of $\oplus_{k\geq0} T_{(k,\alpha)} A_k$  is smaller than the unconditional basis const of $F$ and   the unconditional basis const of $\oplus_{k\geq0} T_{(k,\alpha)} A_k$  is infinity, we have that the unconditional basis const of $F$ is infinity. Thus from proposition \ref{Proposition: Unconditional Matrix}, we know that $F$ is a conditional matrix and $\widetilde{T}U$ is a conditional matrix.

Since the entries of $\widetilde{T}$ is just a rearrangement of $T$, one can easily find an unitary matrix (operator) $\widetilde{U}$ such that $\widetilde{U} \widetilde{T}\widetilde{U}^*=T$, it follows that $\widetilde{U}^*T\widetilde{U}U$ is a conditional matrix. Again from theorem \ref{BPS}, $T\widetilde{U}U$ is a conditional matrix. Thus $T$ is a conditional operator, since it maps orthonormal basis
$\{(\widetilde{U}U)e_1$,$\ldots$,
$(\widetilde{U}U)e_n$,$\ldots\}$ into a conditional basis.
\end{proof}

Now, we come to the main results.

\begin{theorem}\label{4}
Let $T\geq0$ belong to $\mathcal{L}(\mathcal {H})$ which does not admit the eigenvalue zero. If there exists a constant $\delta>1$ such that $$\lim\limits_{t\rightarrow0^+}Card\{[\frac{t}{\delta},t]\cap\sigma(T)\}=\infty,$$ then $T$ is a conditional operator.
\end{theorem}
\begin{proof}
First step, we choose a sequence $\{\lambda_n\}\subseteq\sigma(T)$ satisfying the conditions of lemma \ref{keylemma}.  We will find it by induction.

For $k=1$, $\Delta_1=\{\lambda_{1},\lambda_{2}\}\subseteq\sigma(T)$ and $c_1,d_1$ can be easily chosen such that \begin{eqnarray*}\frac{d_1}{c_1}\leq \delta ~{\rm  and }~ c_1\leq\frac{\alpha}{\lambda_{1}},\frac{\alpha}{\lambda_{2}}\leq d_1.\end{eqnarray*}

Suppose we have found $\Delta_{k-1}=\{\lambda_{2^{k-1}-1},\lambda_{2^{k-1}},\lambda_{2^{k-1}+1},\cdots,\lambda_{2^{k}-2}\}\subseteq\sigma(T)$ which satisfies $$\Delta_{k-1}\cap\bigcup_{1\leq j\leq k-2}\Delta_j=\emptyset,$$ and $c_{k-1},d_{k-1}$ such that the first two conditions of lemma \ref{keylemma} are satisfied. Since $$\lim\limits_{t\rightarrow0^+}Card\{[\frac{t}{\delta},t]\cap\sigma(T)\}=\infty,$$  we can find $t_0<min\{\lambda;~\lambda\in\bigcup_{1\leq j\leq k-1}\Delta_j\}$  such that $t\leq t_0$, $$Card\{[\frac{t}{\delta},t]\cap\sigma(T)\}\geq2^k.$$

Choose arbitrary two elements $\{\lambda_{2^{k+1}-3},\lambda_{2^{k+1}-2}\}\subseteq\sigma(T)\cap[\frac{t_0\alpha^k}{\delta},t_0\alpha^k]$, then choose one after one as follows,
\begin{eqnarray*}
&&\{\lambda_{2^{k+1}-5},\lambda_{2^{k+1}-4}\}\subseteq\{\sigma(T)\cap[\frac{t_0\alpha^{k-1}}{\delta},t_0\alpha^{k-1}]\}\backslash\{\lambda_{2^{k+1}-3},\lambda_{2^{k+1}-2}\}\\ &&\hspace{50mm}\vdots\\
&&\{\lambda_{(2^j-1)2^{k-j+1}-1},\lambda_{(2^j-1)2^{k-j+1}},\lambda_{(2^j-1)2^{k-j+1}+1},\ldots,\lambda_{(2^{j+1}-1)2^{k-j}-2}\}\subseteq\{\sigma(T)\\
&&\cap[\frac{t_0\alpha^{j}}{\delta},t_0\alpha^{j}]\}\backslash\{\lambda_{(2^{j+1}-1)2^{k-j}-1},\lambda_{(2^{j+1}-1)2^{k-j}},\lambda_{(2^{j+1}-1)2^{k-j}+1},\ldots,\lambda_{2^{k+1}-2}\}\\
&&\hspace{50mm}\vdots\\
&&\{\lambda_{2^k-1}, \lambda_{2^k},\ldots,\lambda_{3\cdot2^{k-1}-2}\}\subseteq\{\sigma(T)\cap[\frac{t_0\alpha}{\delta},t_0\alpha]\}\backslash\{\lambda_{3\cdot2^{k-1}-1},\lambda_{3\cdot2^{k-1}},
\ldots,\\
&&\lambda_{2^{k+1}-2}\}.\end{eqnarray*}
Since $Card\{[\frac{t_0\alpha^{j}}{\delta},t_0\alpha^{j}]\cap\sigma(T)\}$ is more than $2^k$, the above process is reasonable.

Denote $c_k=t_0^{-1},d_k=\delta t_0^{-1}$, then obviously \begin{eqnarray*}&&c_k\leq\frac{\alpha}{\lambda_{2^k-1}},\ldots,\frac{\alpha}{\lambda_{3\cdot2^{k-1}-2}},\frac{\alpha^2}{\lambda_{3\cdot2^{k-1}-1}},\ldots,\frac{\alpha^2}{\lambda_{7\cdot2^{k-2}-2}},\\
&&\cdots\cdots,\frac{\alpha^{k-1}}{\lambda_{2^{k+1}-5}},\frac{\alpha^{k-1}}{\lambda_{2^{k+1}-4}},\frac{\alpha^k}{\lambda_{2^{k+1}-3}},\frac{\alpha^k}{\lambda_{2^{k+1}-2}} \leq d_k.
\end{eqnarray*}

Thus we have found a sequence $\{\lambda_n\}\subseteq\sigma(T)$ satisfying the conditions of lemma \ref{keylemma}. Obviously, $\lambda_n$ converges to zero as $n\rightarrow\infty$.

Second step, we will complete the proof.

We rearrange the sequence $\{\lambda_n\}\subseteq\sigma(T)$ into a decreasing sequence $\{\mu_n\}$. Fix a constant $M>\frac{||T||}{\mu_1}$.

For $n\geq1$, cut each segment $[\mu_{n+1},\mu_n]$ into smaller subsegments (many enough and we denote them by $[\nu_{m^n_{j+1}},\nu_{m^n_{j}}],~1\leq j\leq k(n)-1$, $\nu_{m^n_{1}}=\mu_n$, $\nu_{m^n_{k(n)}}=\mu_{n+1}$) in order that $$\dfrac{\nu_{m^n_{j}}}{\nu_{m^n_{j+1}}}\leq M,~1\leq j\leq k(n)-1,n=1,2,\ldots.$$

From the spectral decompose theorem of self-adjoint operator, we have $$T=\oplus_{n\geq0}\oplus_{1\leq j\leq k(n)-1} T_{(n,j)},$$ where $T_{(n,j)}$ is the operator on the subspace $\mathcal {H}_{(n,j)}$ corresponding to $[\nu_{m^n_{j+1}},\nu_{m^n_{j}}]\cap\sigma(T)$ for $n\geq1$ and $T_{(0)}$ is the operator on the subspace $\mathcal {H}_{(0)}$ corresponding to $[\mu_1,\infty)\cap \sigma(T)$.

Denote $$X=\oplus_{n\geq0}\oplus_{1\leq j\leq k(n)-1}\xi_{(n,j)}^{-1}T_{(n,j)},$$ where $\xi_{(0)}=\mu_1$, $\xi_{(n,j)}\in[\nu_{m^n_{j+1}},\nu_{m^n_{j}}]\cap\sigma(T)$ and $\xi_{(n,1)}=\mu_n$. Then since
\begin{eqnarray*}||\xi_{(n,j)}^{-1}T_{(n,j)}||\leq M ~{\rm and}~
||(\xi_{(n,j)}^{-1}T_{(n,j)})^{-1}||\leq M,~1\leq j\leq k(n)-1,~n\geq0,
\end{eqnarray*} we have $X$ is an invertible operator.
Moreover, $$S\triangleq\oplus_{n\geq0}\oplus_{1\leq j\leq k(n)-1}\xi_{(n,j)}I_{(n,j)}=X^{-1}T,$$ where $I_{(n,j)}$ is the identity operator on $\mathcal {H}_{(n,j)}$.
Obviously, $S$ is a diagonal operator with $\{\lambda_n\}$ its subsequence. Thus $S$ satisfies the conditions of lemma \ref{keylemma} and hence it is a conditional operator.
From  theorem \ref{BPS}, we obtain that $T$ is a conditional operator.
\end{proof}

Following is a easier criterion for an operator to be conditional.

\begin{theorem}\label{5}
Let $T\geq0$ belong to $\mathcal{L}(\mathcal {H})$ which does not admit the eigenvalue zero. If $\sigma(T)$ has a decreasing sequence $\{\lambda_n\}$ which converges to zero such that $$\lim\limits_{n\rightarrow\infty}\frac{\lambda_n}{\lambda_{n+1}}=1,$$ then $T$ is a conditional operator.
\end{theorem}
\begin{proof}
It suffices to show that there exists a constant $\delta>1$ such that $$\lim\limits_{t\rightarrow0^+}Card\{[\frac{t}{\delta},t]\cap\{\lambda_n,n\geq1\}\}=\infty.$$

If not, then there exists $N>0$, such that for any $t_0>0$, there is a $t\leq t_0$, $$Card\{[\frac{t}{\delta},t]\cap\{\lambda_n,n\geq1\}\}<N.$$ Thus there exist sequences $a_{k},b_{k}$ converge to zero, such that for all $k$
\begin{eqnarray*}&&\frac{b_{k}}{a_{k}}=\delta, Card\{[a_{k},b_{k}]\cap\{\lambda_n,n\geq1\}\}<N,\\
&&b_{k+1}<a_{k},Card\{[b_{k+1},a_{k}]\cap\{\lambda_n,n\geq1\}\}\geq1.\end{eqnarray*}

Choose $\lambda_{n_1}$ such that $\lambda_{n_1}=min\{\lambda_{n};~\lambda_{n}\geq b_1\}$, choose $\lambda_{n_2}$ such that $\lambda_{n_2}=max\{\lambda_{n};~\lambda_{n}\leq a_1\}$. Generally, choose $\lambda_{n_{2k-1}}=min\{\lambda_{n};~\lambda_{n}\geq b_{k}\}$ and $\lambda_{n_{2k}}=max\{\lambda_{n};~\lambda_{n}\leq a_{k}\}$. It is easy to see that $n_{2k}-n_{2k-1}\leq N$.

On the other hand, since $$\lim\limits_{n\rightarrow\infty}\frac{\lambda_n}{\lambda_{n+1}}=1,$$ we have $$\lim\limits_{n\rightarrow\infty}\frac{\lambda_{n}}{\lambda_{n+j}}=1,$$ for any $1\leq j\leq N$ and hence $$\lim\limits_{k\rightarrow\infty}\frac{\lambda_{n_{2k-1}}}{\lambda_{n_{2k}}}=1.$$ But $$\frac{\lambda_{n_{2k-1}}}{\lambda_{n_{2k}}}\geq\frac{b_{k}}{a_{k}}=\delta>1$$ for any $k$, it is a contradiction.

Thus $T$ is a conditional operator.
\end{proof}

\begin{remark}
Actually, suppose the limit of $\frac{\lambda_n}{\lambda_{n+1}}$ exists, then  $$\lim\limits_{n\rightarrow\infty}\frac{\lambda_n}{\lambda_{n+1}}=1$$ if and only if there exists a constant $\delta>1$ such that $$\lim\limits_{t\rightarrow0}Card\{[\frac{t}{\delta},t]\cap\{\lambda_n,n\geq1\}\}=\infty.$$ One can easily prove it. Thus the condition of theorem \ref{5} is a little stronger than theorem \ref{4}.
\end{remark}

\begin{corollary}
Let $T\in\mathcal{L}(\mathcal {H})$ such that $T$ and $T^*$ do not admit the eigenvalue zero. If $\sigma((T^*T)^{\frac{1}{2}})$ has a decreasing sequence $\lambda_n$ which converges to zero such that $$\limsup\limits_{n\rightarrow\infty}\frac{\lambda_n}{\lambda_{n+1}}=1,$$ then $T$ is a conditional operator.
\end{corollary}
\begin{proof}
From the polar decomposition theorem, $$T=U(T^*T)^{\frac{1}{2}},$$ where $U$ is a unitary operator. Thus from theorem \ref{5} and theorem \ref{BPS}, we obtain the result.
\end{proof}

\begin{corollary}
Compact operator $K=diag\{1,\frac{1}{2},\frac{1}{3},\cdots\}$ is a conditional operator.
\end{corollary}

\end{document}